\documentclass[12pt,fleqn]{article}
\usepackage{amsmath}
\usepackage{amsfonts}
\usepackage{amsthm}

\newcommand{\Rset}{\mathbb{R}}

\newcommand{\F}{\mathcal{F}}

\begin{document}

\title{SELFDECOMPOSABLE LAWS ASSOCIATED WITH HYPERBOLIC FUNCTIONS}

\author{Zbigniew J. Jurek and Marc Yor}

\date{\it{Probab. Math. Stat.} vol. \bf{24.1} (2004)}

\maketitle

\newtheorem{prop}{PROPOSITION}
\newtheorem{cor}{COROLLARY}

\theoremstyle{remark}
\newtheorem{rem}{REMARK}

\begin{quote}
\noindent {\footnotesize \textbf{ABSTRACT.}  It is shown that the
hyperbolic functions can be associated with selfdecomposable
distributions (in short: SD probability distributions or L\'evy
class L probability laws). Consequently, they admit associated
background driving L\'evy processes $Y$ (BDLP $Y$). We interpret
the distributions of $Y(1)$ via Bessel squared processes, Bessel
bridges and local times.}

\medskip
MSC\,2000 \emph{subject classifications.} Primary 60E07, 60E10 ;
secondary 60B11, 60G51.

\medskip
\emph{Key words and phrases:} Hyperbolic characteristic functions;
class L or SD probability distribution; selfdecomposability;
L\'evy process; Bessel squared process; Bessel bridge; local
times.

\end{quote}

\textbf{1. Introduction and terminology.} The aim of this note is
to provide a new way of looking at the hyperbolic functions:
\emph{cosh, sinh} and \emph{tanh}, or their modifications, as the
members of the class SD of selfdecomposable characteristic
functions (often called class L, after Paul L\'evy). Analytically
one says that \emph{a characteristic function  $\phi$ is
selfdecomposable}, we simply write $ \phi \in SD$,if

\begin{equation}
\forall(0<c<1)\, \exists \,\rho_{c}\, \forall(t \in \Rset) \ \
\phi(t)= \phi(ct) \rho_{c}(t),
\end{equation}
where $ \rho_c$ is also a characteristic function. Let us recall
here that class $SD$ is a proper subset of $ID$, the class of all
\emph{infinitely divisible characteristic functions}, and that the
factors $\rho_{c}$ in (1) are in $ID$ as well ; cf. Jurek and
Mason (1993), Section 3.9., or Lo\'eve (1963), Section 23 (there
this class is denoted by $\mathfrak{N}$). We also will use the
convention that a random variable $X$ (in short: r.v. $X$) or its
probability distribution $\mu_{X}$ or its probability density
$f_{X}$ is selfdecomposable if the corresponding characteristic
function is in the class $SD$. Furthermore, the equation (1)
describing the selfdecomposability property, in terms of a r.v.
$X$ means that
\[
X \in SD \quad \mbox{iff} \quad \forall(0<c<1) \exists(r.v.\,\,
X_c) \quad X \stackrel{d}{=} cX+X_{c},
\]
where the r.v. $X$ and $X_c$ are independent and $\stackrel{d}{=}$
means equality in distribution.

\medskip

For further references let us recall the main properties of the
selfdecomposable distributions (or characteristic functions or
r.v's):

\begin{description}
\item[(a)] $SD$ \emph{with the convolution and the weak convergence forms a
closed
convolution subsemigroup of} $ID$;

\item[(b)] $SD$ \emph{is closed under affine mappings, i.e., for all reals a
and b one has}: $\phi \in SD$ iff  $e^{ibt}\phi(at) \in SD$.

\item[(c)] $X \in SD$ \emph{iff there exists  a (unique) L\'evy
process
$Y(\cdot)$ such that} \\
$X \stackrel{d}{=} \int_{0}^{\infty}e^{-s}dY(s)$, \emph{where Y is
called the BDLP (background driving \\
L\'evy process) of  the X. Moreover, one has that}
$\mathbb{E}[\log(1+|Y(1)|)] < \infty$.

[$ID_{\log}$ will stand for the class of all infinitely divisible
laws with finite logarithmic moments.]

\item[(d)] \emph{Let $\phi$ and $\psi$ denote the characteristic function of
$X$ and $Y(1)$, respectively, in part (c). Then one has} \\
$\log \phi(t) = \int_{0}^{t} \log \psi(v) \frac{dv}{v}$, i.e.,
$\psi(t)= \exp[t(\log \phi(t))^{'}],t \neq 0, \psi(0)=1$.

\item[(e)] \emph{Let M be the L\'evy spectral measure in the L\'evy-Khintchine
formula of $\phi \in SD$. Then M has a density $h(x)$ such that $xh(x)$
is non-increasing on the positive and negative half-lines. \\
Furthermore, if $h$ is differentiable almost everywhere then
$dN(x) = -(xh(x)^{'}dx$ is the L\'evy spectral measure of $\psi$ in (d). \\
Finally, one has also the following logarithmic moment
condition }\\
$\int_{\{|x| \ge \epsilon\}} \log(1+|x|)dN(x) < \infty$,\quad
\emph{for all positive $\epsilon$}.
\end{description}

Parts \textbf{(a)} and \textbf{(b)} follow directly from (1). For
\textbf{(c)} and \textbf{(d)} cf. Jurek \& Mason (1993), Theorem
3.6.8 and Remark 3.6.9(4). Part \textbf{(e)} is Corollary 1.1 from
Jurek (1997).

In this note we will characterize the BDLP's (or the
characteristic functions $\psi$ in \textbf{d}) for the hyperbolic
characteristic functions. The main result shows how to interpret
these distributions in terms of squared Bessel bridges (Corollary
2) and squaredd Bessel processes (Corollary 3).

\textbf{2. Selfdecomposability of the hyperbolic characteristic
functions.} For this presentation the most crucial example of $SD$
r.v. is that of the \emph{Laplace (or double exponential) random
variable} $\eta$. So, $\eta$ has the probability density
$\frac{1}{2}e^{-|x|}, x \in \Rset$, and its characteristic
function is equal
\begin{multline}
\phi_{\eta}(t)= \frac{1}{1+t^2}= \exp \Big[\int_{-
\infty}^{\infty}(e^{itx}-1)
\frac{e^{-|x|}}{|x|}\,dx \Big] \\
= \exp \int_{0}^{t} \Big[ \int_{-
\infty}^{\infty}(e^{ivx}-1)e^{-|x|}dx \Big] \frac{dv}{v} \in SD.
\quad \quad \quad \quad \quad
\end{multline}
To see its selfdecomposability property simply note that
\[
\frac{\phi_{\eta}(t)}{\phi_{\eta}(ct)} = \frac{1+c^2t^2}{1+t^2} =
c^2 1 + (1-c^2)\frac{1}{1+t^2} \quad \mbox{is the characteristic
function of}\, \rho_{c},
\]
in the formula (1). The rest follows from appropriate
integrations; cf. Jurek (1996).

Another, more "stochastic" argument for selfdecomposability of
Laplace rv $\eta$, as a counterpart to the above analytic one, is
as follows.

Firstly, notice that for three independent rvs $\mathcal{E}(1)$,
$\Tilde{\mathcal{E}}(1)$ and $b_c$, where the first two have
exponential distribution with parameter 1 and the third one has
Bernoulli distribution ($P(b_c=1)= 1-c$ and $P(b_c=0)=c$), one has
equality
\[
\mathcal{E}(1)\stackrel{d}{=} c \mathcal{E}(1)+ b_c
\Tilde{\mathcal{E}}(1),
\]
which means that $\mathcal{E}(1)$ is a selfdecomposable rv. (The
above distributional equality is easily checked by using the
Laplace or Fourier transform.)

Secondly, taking two independent Brownian motions $B_t, \,
\Tilde{B}_t, \, t\ge 0$ and independently of them an exponential
rv $\mathcal{E}(1)$ satisfying the above decomposition, we infer
that
\[
B_{\mathcal{E}(1)} \stackrel{d}{=} \sqrt{c} B_{\mathcal{E}(1)} +
\Tilde{B}_{b_c \Tilde{\mathcal{E}}(1)},
\]
and thus proving that stopped Brownian motion $B_{\mathcal{E}(1)}$
is selfdecomposable as well.

Thirdly, let us note that $B_{\mathcal{E}(1)}$ has the double
exponential distribution. More explicitly we have
\[
\mathbb{E}[e^{it(\sqrt{2}B_{\mathcal{E}(1)})}] =
\mathbb{E}[e^{-t^2 \mathcal{E}(1)}] = \frac{1}{1+t^2}= \phi_{\eta}(t).
\]
For more details and a generalization of this approach cf. Jurek
(2001), Proposition 1 and Bondesson (1992), p. 19.

\begin{prop}
The following three hyperbolic functions: $\frac{1}{\cosh t},
\quad \frac{t}{\sinh t}$, $\frac{\tanh t}{t}, t \in \Rset$, are
characteristic functions of selfdecomposable probability
distributions, i.e., they are in the class SD.
\end{prop}
\begin{proof}
From the following product representations :
\begin{equation}
\cosh z = \prod_{k=1}^{ \infty}(1+ \frac{4z^2}{(2k-1)^{2}
\pi^{2}}), \quad \sinh z = z \prod_{k=1}^{ \infty}(1+
\frac{z^2}{k^2 \pi^{2}}),
\end{equation}
for all complex z, and from (2) with \textbf{(a)} we conclude that
the first two hyperbolic functions are characteristic functions
from SD. Moreover, these are characteristic functions of series of
independent Laplace r.v.; cf. Jurek (1996).

Note that for $0<a<b$ the fraction
\[
\frac{1+a^2t^2}{1+b^2t^2} = \frac{a^2}{b^2} 1 + (1-
\frac{a^2}{b^2}) \frac{1}{1+b^2t^2} \quad \mbox{is a
characteristic function,}
\]
\[
\mbox{and so is} \quad \frac{\tanh t}{t} =
\prod_{k=1}^{\infty}\frac{1+(k\pi)^{-2}t^2}{1+((k-\frac{1}{2})\pi)^{-2}t^2},
\]
as a converging infinite series of characteristic functions of the
above form. Its selfdecomposbility follows from Yor (1997), p.
133, or Jurek (2001), Example 1(b).
\end{proof}

\begin{rem}
The selfdecomposability of $\frac{\tan t}{t}$, i.e., the formula
(1), would follow in an elementary manner if for all $0<c<1$ and
all $0<w<u$, the functions
\[
\frac{1+(c^{2}u +w)t^2 +c^2uw t^4}{1+(u+c^2w)t^2 +c^2uw t^4}=
\frac{1+wt^2}{1+ ut^2}  \frac{1+ c^2ut^2}{1+c^2wt^2} \quad
\mbox{were characteristic functions. }
\]
[The above is a ratio of \emph{two} fractions of the form as in
the product representation of $\tanh t /t$, with the fraction in
the denominator computed at $ct$.] However, they \emph{can not be}
characteristic functions ! Affirmative answer would mean that
Laplace rv is in $L_1$ (these are those $SD$ rv for which  $BDLP$
$Y(1)$ in \textbf{(c)}, is $SD$). Equivalently, the characteristic
function $\rho_{c}$ in (1) is in $SD$).) But from (2) we see that
$Y(1)$ for  Laplace rv $\eta$ has compound Poisson distribution
with L\'evy spectral measure $dM(x) = e^{-|x|}dx$ which does not
satisfy the criterium \textbf{(e)}. Cf. also Jurek ( 1997).
\end{rem}

\medskip
\textbf{3. The BDLP's of the hyperbolic characteristic functions.}
Since the three hyperbolic characteristic functions are infinitely
divisible one can insert them into L\'evy processes:
$\hat{C_{s}}$, $\hat{S_{s}}$, $\hat{T_{s}}$, for $s \ge 0$,
corresponding to \emph{cosh, sinh} and \emph{tanh} characteristic
functions. Those processes were studied from the ID class point of
view in the recent paper Pitman-Yor (2003). Here we are looking at
them from the SD class point of view, i.e., via the corresponding
BDLP's.

Below $\phi$ with subscript $\hat{C}, \hat{S}$ or $\hat{T}$
denotes one of the three hyperbolic characteristic functions, $M$
with similar subscripts denotes the L\'evy spectral measure in the
appropriate L\'evy-Khintchine formula, furthermore $\psi$  with
the above subscripts is the corresponding characteristic function
in the random integral representation (properties \textbf{(c)} and
\textbf{(d)} of class SD) and finally $N$ with one of the above
subscripts is the L\'evy spectral measure of $\psi$ (as in
\textbf{(e)}). Thus we have the equalities :
\begin{equation}
\phi_{\hat{C}}(t)=\phi_{\hat{S}}(t) \cdot \phi_{\hat{T}}(t), \quad
\mbox{i.e.,} \quad
\frac{1}{\cosh t}= \frac{t}{\sinh t} \cdot \frac{\tanh t}{t},
\end{equation}

\begin{multline}
M_{\hat{C}}(\cdot) =M_{\hat{S}}(\cdot) + M_{\hat{T}}(\cdot), \quad
\mbox{where} \quad
\frac{dM_{\hat{C}}(x)}{dx} = \frac{1}{2x\sinh (\pi x/2)}; \\
\frac{dM_{\hat{S}}(x)}{dx} = \frac{e^{-\pi|x|/2}}{2x\sinh (\pi x/2} =
\frac{1}{2|x|} (\coth (\frac{\pi|x|}{2}) -1); \\
\frac{dM_{\hat{T}}(x)}{dx} = \frac{1}{2|x|}
\frac{e^{-\pi|x|/4}}{\cosh(\pi|x|/4)}= \frac{1}{2|x|}[1- \tanh(\pi|x|/4)].
\end{multline}
These are consequences of the appropriate L\'evy-Khintchine
formulas for the hyperbolic characteristic functions or see Jurek
(1996) or Pitman-Yor (2003) or use (3) and the product formulas
for $\cosh z, \sinh z$; (for $\tanh z$ use the ratio of the two
previous formulas).

\begin{cor}
For the $SD$ hyperbolic characteristic functions $\phi_{\hat{C}}$,
$\phi_{\hat{S}}$ and $\phi_{\hat{T}}$, their background driving
characteristic functions are $\psi_{\hat{C}}$, $\psi_{\hat{S}}$
and $\psi_{\hat{T}}$, where
\begin{multline}
\psi_{\hat{C}}(t)=\psi_{\hat{S}}(t) \cdot \psi_{\hat{T}}(t); \quad
\psi_{\hat{C}}(t)=\exp[-t \tanh t], \,\,
\psi_{\hat{S}}(t)=\exp[1-t \coth t], \\
\psi_{\hat{T}}(t)=\exp\Big[\frac{1}{\cosh t}\cdot\frac{t}{\sinh
t}- 1 \Big] = \exp\Big[\frac{2t}{\sinh(2t)} - 1 \Big]. \quad \quad
\quad \quad
\end{multline}
Probability distributions corresponding to $\psi_{\hat{C}}$,
$\psi_{\hat{S}}$ and $\psi_{\hat{T}}$ are infinitely divisible
with finite logarithmic moments.
\end{cor}
Proofs follow from (4) and  the properties \textbf{(c)} and
\textbf{(d)} of the  selfdecomposable distributions.

Let us note that  $\psi_{\hat{T}}$ is the characteristic function
of the compound Poisson distribution with summand being the sum of
independent rv's with the $cosh$ and $sinh$ characteristic
functions.

Finally, on the level of the L\'evy measures $N$ of $Y(1)$, from
the $BDLP's$ in the property \textbf{(d)}, we have the following :
\begin{multline}
N_{\hat{C}}(\cdot)= N_{\hat{S}}(\cdot) + N_{\hat{T}}(\cdot), \quad
\mbox{where} \quad
\frac{dN_{\hat{C}}(x)}{dx} = \frac{\pi}{4}
\frac{\cosh(\frac{\pi x}{2})}{\sinh^{2}(\frac{\pi x}{2})}, \\
\frac{dN_{\hat{S}}(x)}{dx}= \frac{\pi}{4}\frac{1}{\sinh^{2}(\frac{\pi x}{2})};
\quad
\frac{dN_{\hat{T}}(x)}{dx}=\frac{\pi}{8}\frac{1}{\cosh^{2}(\frac{\pi x}{4})}.
\end{multline}
Explicitly, as in (4), on the level of the BDLP one has
factorization
\begin{equation}
\exp[-t \tanh t]= \exp[1-t \coth t] \cdot
\exp \Big[ \frac{t}{\cosh t \sinh t}-1 \Big].
\end{equation}

Taking into account all the above and the property \textbf{(d)} we
arrive at the identities :
\begin{equation}
\int_{R \backslash \{0\}}(1-\cos tx)\,\frac{\pi}{4}
\frac{\cosh(\pi x/2)}{\sinh^{2}(\pi x/2)}\,dx= t \tanh t ;
\end{equation}
\begin{equation}
\int_{R \backslash \{0\}}(1- \cos tx)\,\frac{\pi}{4}
\frac{1}{\sinh^{2}(\pi x/2)}\,dx=t \coth t -1 ;
\end{equation}
\begin{equation}
\int_{R \backslash \{0\}}(1- \cos tx)\,\frac{\pi}{8}
\frac{1}{\cosh^{2}(\pi x/4)}\,dx= 1- \frac{2t}{\sinh 2t}.
\end{equation}
Furthermore, since $\psi_{\hat{T}}$ corresponds to a compound
Poisson distribution, the last equality implies that
\begin{equation}
\int_{R \backslash \{0\}}\cos tx\, [\frac{\pi}{8}
\frac{1}{\cosh^{2}(\pi x/4)}]\,dx= \frac{2t}{\sinh 2t},
\end{equation}
where we recover the known relation between $(\cosh u)^{-2}$ being
the probability density corresponding to the characteristic
function $\frac{at}{\sinh at}$ and vice versa by the inversion
formula; cf. P. L\'evy (1950) or Pitman and Yor (2003), Table 6.

\medskip

\textbf{4. Stochastic interpretation of BDLP's for hyperbolic
functions.} The functions $\psi_{\hat{C}}(t)$ and
$\psi_{\hat{S}}(t)$ were identified as
 characteristic functions of the \emph{background driving random variable} $Y(1)$ (in short:
BDRV) for $cosh$ and $sinh$ SD rv in Jurek (1996), p. 182. [By the
way, the question raised there has an affirmative answer. More
precisely: (23) implies (24). To see that note that $D_1
\stackrel{d}{=} \frac{1}{2} \tilde{D}_{1}+C$, where $\tilde{D}_1$
is a  copy of $D_1$ and independent of $C$ . The notations here
are from the paper in question.]

More recently, in Jurek (2001) it was noticed that the conditional
characteristic function of the L\'evy's stochastic area integral
is a product of the $sinh$ characteristic function and its BDLP
$\psi_{\hat{S}}$. Similar factorization one has in Wenocur
formula; Wenocur (1986). Cf. also Yor (1992a), p. 19.

In this section we give some \emph{"stochastic"} interpretation of
the characteristic functions $\psi_{\hat{C}}(t)$ and
$\psi_{\hat{S}}(t)$ in terms of Bessel processes.

Let us recall here some basic facts and notations from Pitman and
Yor (1982) and Yor (1992a, 1997). Also cf. Revuz and Yor (1999).

For $\delta$-dimensional Brownian motion $(B_t , t \ge 0)$,
starting from  a vector $a$, we define the process $X_t = |B_t|^2,
t \ge 0$, which in turn defines the probability distribution (law)
$Q^{\delta}_{x},$\,$ x:=|a|^2$, on the canonical space $\Omega:=
C([0, \infty);[0,\infty))$ of non-negative functions defined on
the half-line $[0,\infty)$, equipped with the $\sigma$-field $\F$
such that mappings $\{\omega \to X_{s}(\omega)\}$ are measurable.
In fact, $(X_t, t\ge0)$ is the unique  strong solution of a
stochastic integral equation
\[
X_t=x+2\int_0^t \sqrt{X_s}\,d\beta_s+\delta \, t, \ \ t\ge 0,
\]
where $(\beta_t, t \ge 0)$ is 1-dimensional Brownian motion.

The laws $Q^{\delta}_{x}$ satisfy the following convolution
equation due to Shiga-Watanabe:
\begin {equation}
Q^{\delta}_{x} \star Q^{\delta^{'}}_{x^{'}} =
Q^{\delta + \delta^{'}}_{x+x^{'}} \quad \mbox{for all} \quad \delta,
\delta^{'}, x , x^{'} \ge 0,
\end{equation}
where, for $P$ and $Q$ two probabilities on $(\Omega,\F)$, $P
\star Q$ denotes the distribution of $(X_t +Y_t, t \ge 0)$, with
$(X_t, t \ge 0)$ and $(Y_t, t \ge 0)$ two independent processes,
respectively $P$ and $Q$ distributed; cf. Revuz and Yor (1999),
Chapter XI, Theorem 1.2.

Similarly, let $Q^{\delta}_{x \to y}$ be $\delta$-dimensional
squared Bessel bridge of $(X_s, 0 \le s \le 1)$, given $X_1 =y$,
viewed as a probability on $C([0,1], [0,\infty))$.

Below we use integrals of functionals $F$ with respect to measures
$Q$ over function spaces. To simplify our notation, as in Revuz
and Yor (1999), we use $Q(F)$ to denote such integrals. From Yor
(1992a, Chapter 2), the L\'evy's stochastic area formula is given
in the form
\begin{equation}
Q^{\delta}_{x \to 0}
\Big[\exp(-\frac{\lambda^2}{2}\int_{0}^{1}dsX_s)\Big]=\Big(
\frac{\lambda}{\sinh \lambda}\Big)^{\delta/2}\exp
\Big(-\frac{x}{2}(\lambda \coth \lambda -1)\Big)
\end{equation}
However, since $Q^{\delta}_{x \to 0}=Q^{\delta}_{0 \to 0} \star
Q^{0}_{x \to 0}$, (cf. Yor (1992), Pitman and Yor (1982)) we have
in fact that
\[
Q^{0}_{x \to 0}\Big(\exp(-\frac{\lambda^{2}}{2}\int_{0}^{1} ds
X_s)\Big) = \exp\Big(- \frac{x}{2}\,(\lambda \coth \lambda
-1)\Big)
\]
Thus we may conclude the following
\begin{cor}
The $BDLP$ $Y$, for the $SD$ characteristic function
$\phi_{\hat{C}}(t)=\frac{t}{\sinh t}$, is such that $Y(1)$ has the
characteristic function
\begin{equation}
\psi_{\hat{S}}(t) = \exp(1-t \coth t) = Q^{0}_{2 \to
0}\Big(\exp(it \gamma_{(\int_{0}^{1}dsX_s)}) \Big)\in ID_{\log},
\end{equation}
where $(\gamma_s, s \ge 0)$ is a Brownian motion independent of
the Bessel squared process X.

[Here it may be necessary to enlarge the probability space to
support independent $\gamma$ and $X$.]
\end{cor}

In a similar way, in view of Yor (1992a), Chapter 2, we have
\[
Q^{\delta}_{x} \Big(\exp(- \frac{\lambda^2}{2} \int_{0}^{1} dsX_s)
\Big) = \Big(\frac{1}{\cosh \lambda}\Big)^{\delta /2}\,\exp
(-\frac{x}{2}\, \lambda \, \tanh \lambda ),
\]
so, in particular,
\[
Q^{0}_{x} \Big(\exp(- \frac{\lambda^2}{2} \int_{0}^{1} dsX_s)
\Big) = \exp \Big(-\frac{x}{2} \lambda \tanh \lambda \Big).
\]
Thus, as above, we conclude the following:

\begin{cor}
The $BDLP$ $Y$, for the $SD$ characteristic function
$\frac{1}{\cosh t}$, is such that $Y(1)$ has characteristic
function
\begin{equation}
\psi_{\hat{C}}(t) = \exp (-t \tanh t)= Q^{0}_{2} \Big(\exp (it
\gamma_{(\int_{0}^{1}dsX_s)} ) \Big) \in ID_{\log},
\end{equation}
where a process $(\gamma_s, s \ge 0)$ is a Brownian motion
independent of the Bessel squared process $X$.
\end{cor}
\medskip
Let us return again to functions $\psi_{\hat{C}}(t)$ and
$\psi_{\hat{S}}(t)$, given in (6), but viewed this time as
Laplace transforms in $t^{2}/2$. From Yor (1997), p. 132, we have
\begin{equation}
\frac{1}{\cosh t}= \mathbb{E}\Big[\exp(- \frac{t^2}{2}\,T^{(1)}_1)
\Big], \quad \ \ \frac{t}{\sinh t}= \mathbb{E}\Big[\exp(-
\frac{t^2}{2}\,T^{(3)}_1 ) \Big],
\end{equation}
where $T^{(\delta)}_{1}:= \inf \{t: \mathcal{R}^{(\delta)}_t = 1
\}$ denotes the hitting time of 1 by $\delta$-dimensional Bessel
process $\mathcal{R}^{(\delta)}_t, \, t \ge0,$ starting from zero.
Jeanblanc-Pitman-Yor (2002), Theorem 3, found that the
corresponding BDLP's $Y$ are of the form
\begin{equation}
Y(h)= \int_{0}^{\tau_{h}^{(\delta)}}du\,
1_{(\mathcal{R}_{u}^{(\delta)} \le 1)}, \quad \quad h \ge 0,
\end{equation}
where $(\tau_{h}^{r}, h \ge 0)$ is the inverse of the local time
of $\mathcal{R}_{u}^{(\delta)}$ at $r$; cf. Revuz and Yor (1999),
Chapter VI, for all needed notion and definitions. From the above
we also recover the formulae
\begin{multline}
\mathbb{E}\Big[\exp(- \frac{\lambda^{2}}{2}
\int_{0}^{\tau_{1}^{(1)}}du 1_{(|B_u| \le 1)})\Big] =
\exp (- \lambda \tanh \lambda), \\
\mathbb{E}\Big[ \exp(- \frac{\lambda^{2}}{2}
\int_{0}^{\tau_{1}^{(3)}}du 1_{(\mathcal{R}_{u}^{(3)} \le 1)})
\Big]= \exp(- \lambda (\coth \lambda -1)).
\end{multline}
These as well provide another "stochastic view" of the analytic
formulae for the $BDRV$ of two $SD$ hyperbolic characteristic
functions in (4), i.e., $1/\cosh t$ and $t/\sinh t$.
\medskip

\begin{center}
\textbf{REFERENCES}
\end{center}

\noindent L. Bondesson (1992), \emph{Generalized gamma
convolutions and related classes of distributions and densities}.
Lect. Notes in Statist., vol. 76, Springer-Verlag, New York.

\medskip
\noindent M. Jeanblanc, J. Pitman and M. Yor (2002), Self-similar
processes with independent increments associated with L\'evy and
Bessel processes.\ \emph{Stoch. Proc. Appl.} vol.100, pp. 223-232.

\medskip
\noindent Z. J. Jurek (1996), Series of independent exponential
random variables. In: \emph{Proc. $7^{th}$ Japan-Russia Symposium
on Probab. Ther. Math. Stat.}; S. Watanabe, M. Fukushima, Yu.V.
Prohorov, and A.N. Shiryaev Eds, pp. 174-182. World Scientific,
Singapore, New Jersey.

\medskip
\noindent Z. J. Jurek (1997), Selfdecomposability: an exception or
a rule ? \emph{Annales Univer. M. Curie-Sk\l odowska,
Lublin-Polonia}, vol. LI, Sectio A, pp. 93-107. (A Special volume
dedicated to Professor Dominik Szynal.)

\medskip
\noindent Z. J. Jurek (2001), Remarks on the selfdecomposability
and new examples, \  \emph{Demonstratio Math.} vol. XXXIV(2), pp.
241-250. (A special volume dedicated to Professor Kazimierz
Urbanik.)

\medskip
\noindent Z. J. Jurek and J. D. Mason (1993), \emph{Operator limit
distributions in probability theory}, J. Wiley and Sons, New York.

\medskip
\noindent P. L\'evy (1951), Wiener's random functions, and other
Laplacian random functions; \emph{Proc. 2nd Berkeley Symposium on
Math. Stat. Probab.}, Univ. California Press, Berkeley, pp.
171-178.

\medskip
\noindent M. Lo\'eve (1963), \emph{Probability theory}, D. van
Nostrand Co., Princeton, New Jersey.

\medskip
\noindent J. Pitman and M. Yor (1981), Bessel processes and
infinite divisible laws. In: \emph{Stochastic Integrals; Proc. LMS
Durham Symposium 1980}. Lect. Notes in Math. vol. 851, pp.
285-370.

\medskip
\noindent
J. Pitman and M. Yor (1982), A decomposition of Bessel bridges, \emph{Z.
Wahrscheinlichkeistheorie verw. Gebiete}, vol. 59, pp. 425-457.

\medskip
\noindent J. Pitman and M. Yor (2003a), Infinitely divisible laws
associated with hyperbolic functions, \emph{Canadian J. Math.},
vol.55 (2), pp. 292-330.

\medskip
\noindent J. Pitman and M. Yor (2003b), Hitting, occupation and
local times of one-dimensional diffusions: martingale and
excursion approaches, \emph{Bernoulli}, vol. 9 no. 1, pp. 1-24.

\medskip
\noindent D. Revuz and M. Yor (1999), \emph{Continuous Martingales
and Brownian Motion}, Springer-Verlag, Berlin-Heildelberg, 3rd
edition.

\medskip
\noindent M. Wenocur (1986), Brownian motion with quadratic
killing and some implications. \emph{J. Appl. Probab.} 23, pp.
893-903.

\medskip
\noindent M. Yor (1992), Sur certaines fonctionnelles
exponentielles du mouvement Brownien reel. \emph{J. Appl. Prob.
29}, pp. 202-208.

\medskip
\noindent M. Yor (1992a), \emph{Some aspects of Brownian motion},
Part I : Some special functionals. Birkhauser, Basel.

\medskip
\noindent M. Yor (1997), \emph{Some aspects of Brownian motion},
Part II : Some recent martingale problems. Birkhauser, Basel.

\medskip
\medskip
\noindent
Institute of Mathematics\hfill Laboratoire de Probabilit\'es \\
University of Wroc\l aw\hfill Universit\'e Pierre et Marie Curie \\
Pl.Grunwaldzki 2/4\hfill 175, rue du Chevaleret, \\
50-384 Wroclaw, Poland\hfill 75013 Paris, France

\end{document}